\documentclass[a4paper]{amsart}
\usepackage[utf8]{inputenc}
\usepackage[english]{babel}
\usepackage[hidelinks]{hyperref}
\usepackage{amssymb}

\theoremstyle{definition}

\newtheorem*{remark*}{Remark}
\newtheorem*{claim*}{Claim}

\theoremstyle{plain}
\newtheorem{theorem}{Theorem}[section]
\newtheorem{corollary}[theorem]{Corollary}

\newtheorem{observation}[theorem]{Observation}
\newtheorem{lemma}[theorem]{Lemma} 
\newtheorem{conjecture}[theorem]{Conjecture}
\newtheorem{question}[theorem]{Question}
\newtheorem{problem}[theorem]{Problem}

\theoremstyle{remark}
\newtheorem{remark}[theorem]{Remark}

\def\str#1{\mathbf {#1}}

\def\Aut{\mathop{\mathrm{Aut}}\nolimits}
\def\Sym{\mathop{\mathrm{Sym}}\nolimits}
\def\Age{\mathop{\mathrm{age}}\nolimits}

\def\eppa{\mathop{\mathrm{eppa}}\nolimits}
\def\ceppa{\mathop{\mathrm{ceppa}}\nolimits}

\begin{document}
\bibliographystyle{alpha}
\title{EPPA numbers of graphs}

\authors{
\author[D. Bradley-Williams]{David Bradley-Williams}
\address{Institute of Mathematics\\Czech Academy of Sciences\\ Prague, Czech Republic}
\email{williams@math.cas.cz}
\author[P. J. Cameron]{Peter J. Cameron}
\address{School of Mathematics and Statistics\\University of St Andrews\\ UK}
\email{pjc20@st-andrews.ac.uk}
\author[J. Hubi\v cka]{Jan Hubi\v cka}
\address{Department of Applied Mathematics (KAM)\\ Charles University\\ Prague, Czech Republic}
\email{hubicka@kam.mff.cuni.cz}
\author[M. Kone\v cn\'y]{Mat\v ej Kone\v cn\'y}
\address{Institute of Algebra\\TU Dresden\\ Dresden, Germany}
\email{matej.konecny@tu-dresden.de}

\thanks{D. B.-W. was supported by the project EXPRO 20-31529X of the Czech Science Foundation (GAČR) and by the Czech Academy of Sciences CAS (RVO 67985840). J. H. was supported by the project 21-10775S of  the  Czech  Science Foundation (GA\v CR) in the earlier stages of this project, and by a project that has received funding from the European Research Council under the European Union's Horizon 2020 research and innovation programme (grant agreement No 810115) in the later stages. M. K. was supported by the European Research Council (Project POCOCOP, ERC Synergy Grant 101071674). Views and opinions expressed are however those of the authors only and do not necessarily reflect those of the European Union or the European Research Council Executive Agency. Neither the European Union nor the granting authority can be held responsible for them.}

}

\begin{abstract}
If $G$ is a graph, $A$ and $B$ its induced subgraphs, and $f\colon A\to B$ an isomorphism, we say that $f$ is a \emph{partial automorphism} of $G$. In 1992, Hrushovski proved that graphs have the \emph{extension property for partial automorphisms} (\emph{EPPA}, also called the \emph{Hrushovski property}), that is, for every finite graph $G$ there is a finite graph $H$, an \emph{EPPA-witness} for $G$, such that $G$ is an induced subgraph of $H$ and every partial automorphism of $G$ extends to an automorphism of $H$.

The \emph{EPPA number} of a graph $G$, denoted by $\eppa(G)$, is the smallest number of vertices of an EPPA-witness for $G$, and we put $\eppa(n) = \max\{\eppa(G) : \lvert G\rvert = n\}$. In this note we review the state of the area, prove several lower bounds (in particular, we show that $\eppa(n)\geq \frac{2^n}{\sqrt{n}}$, thereby identifying the correct base of the exponential) and pose many open questions. We also briefly discuss EPPA numbers of hypergraphs, directed graphs, and $K_k$-free graphs.
\end{abstract}

\maketitle

\section{Introduction}
Let $\str A$ and $\str B$ be finite structures (e.g. graphs, hypergraphs, or metric spaces) such that $\str A$ is a substructure of $\str B$ (all substructures and subgraphs in this note will be induced). We say that $\str B$ is an \emph{EPPA-witness} for $\str A$ if every isomorphism of substructures of $\str A$ (a \emph{partial automorphism of $\str A$}) extends to an automorphism of $\str B$. We say that a class $\mathcal C$ of finite structures has the \emph{extension property for partial automorphisms} (\emph{EPPA}, also called the \emph{Hrushovski property}) if for every $\str A\in \mathcal C$ there is $\str B\in \mathcal C$ which is an EPPA-witness for $\str A$.

In 1992, Hrushovski~\cite{hrushovski1992} established that the class of all finite graphs has EPPA. This result was used by Hodges, Hodkinson, Lascar, and Shelah to show the small index property for the random graph~\cite{hodges1993b}. After this, the quest of identifying new classes of structures with EPPA continued with a series of papers 
including~\cite{Herwig1995,herwig1998,herwig2000,hodkinson2003,solecki2005,vershik2008,Aranda2017,Hubicka2017sauer,Conant2015,HubickaSemigenericAMUC,Hubicka2018metricEPPA,Konecny2018b,eppatwographs,otto2017,Evans3,Hubicka2018EPPA,HubickaSemigeneric}.
Nevertheless, there are still many classes for which EPPA is open. The most notable examples are the class of all finite tournaments (see~\cite{herwig2000}), the class of all finite partial Steiner systems (see~\cite{Hubicka2018EPPA}), and the class of all finite graphs where automorphisms can also send edges to non-edges and vice versa (see~\cite{eppatwographs}). See also Remark~\ref{rem:homogeneous}.

In this note we will, however, mainly concern ourselves with graphs. Given a graph $G$, we define its \emph{EPPA number} as $$\eppa(G) = \min \{\lvert H\rvert : H\text{ is an EPPA-witness for }G\},$$ and we put $\eppa(n) = \max\{\eppa(G) : \lvert G\rvert = n\}$.

Hrushovski in his paper proved that
$$2^{\frac{n}{2}} \leq \eppa(n) \leq (2n2^n)!$$
and asked if the upper bound can be improved (in particular, if it can be improved to $2^{\mathcal O(n^2)}$). This was answered by Herwig and Lascar in 2000~\cite{herwig2000} who proved that $\eppa(n) \leq \left(\frac{3en}{4}\right)^n$, and that if $G$ is a graph on $n$ vertices with maximum degree $d$ then $\eppa(G)\leq \mathcal O((nd)^d)$. In 2018, another construction of EPPA-witnesses was found by Evans, Hubička, Konečný and Nešetřil, giving an upper bound $\eppa(n)\leq n2^{n-1}$~\cite{eppatwographs,Hubicka2018EPPA} (this construction has been independently discovered also by Andr\'eka and N\'emeti~\cite{Andreka2019}).

In this note we prove several new lower bounds, give a brief review of some of the best currently known upper bounds and techniques giving them, and pose some open problems. 
We prove:
\begin{theorem}\label{thm:lower_bound}
For every $n$, there is a graph $G$ on $n$ vertices with $\eppa(G)\geq {n-1 \choose \lfloor (n-1)/2\rfloor}$. Consequently,
$$\Omega(2^n/\sqrt{n}) \leq \eppa(n) \leq n2^{n-1}.$$
\end{theorem}
\begin{theorem}\label{thm:bwc}
$\eppa(G)\geq \frac{5}{4}\lvert G\rvert$, unless $G$ is a subgraph of a homogeneous graph.
\end{theorem}
\begin{theorem}\label{thm:degrees}
Let $G$ be a graph with $n$ vertices and maximum degree $d$ which is not a subgraph of a homogeneous graph. Then $\eppa(G)\in \Omega((n/d)^2)$. If $G$ is triangle-free then $\eppa(G) \geq {\lceil n/(d+1)\rceil \choose d}$.
\end{theorem}
This theorem in particular implies that $\eppa(C_n) \in \Theta(n^2)$ where $C_n$ is the cycle on $n$ vertices. See Corollary~\ref{cor:cycles} for a more precise statement.

Let $G(n,p)$ denote the random graph on $n$ vertices where each edge is present independently with probability $p$. In particular, $G(n,1/2)$ is the uniform distribution on all graphs on $n$ vertices.
\begin{theorem}\label{thm:random}
For every $\delta > 0$, asymptotically almost surely we have that $$\eppa(G(n,1/2)) \in \Omega(n^{2-\delta}).$$
\end{theorem}
On the other hand, $G(n,c/n)$ is a distribution on graphs on $n$ vertices where the average degree is equal to $c$.
\begin{theorem}\label{thm:random_bddegree}
For every $c,d > 0$, asymptotically almost surely we have that $$\eppa(G(n,c/n)) \in \Omega(n^{d}).$$
\end{theorem}

\medskip

In the statements above we were using asymptotic notation. For completeness, we briefly review it: Given a function $f\colon \mathbb N\to \mathbb N$, we define the following classes of functions:
\begin{align*}
	\mathcal O(f(n)) &= \{ g\colon \mathbb N \to \mathbb N : (\exists c>0)(\forall n)g(n) \leq cf(n)\},\\
	\Omega(f(n)) &= \{ g\colon \mathbb N \to \mathbb N : (\exists c>0)(\forall n)g(n) \geq cf(n)\},\\
	\Theta(f(n)) &= \mathcal O(f(n)) \cap \Omega(f(n)),\\
	o(f(n)) &= \{ g\colon \mathbb N \to \mathbb N : \lim_{n\to \infty} \frac{g(n)}{f(n)} = 0\},\\
	\omega(f(n)) &= \{ g\colon \mathbb N \to \mathbb N : \lim_{n\to \infty} \frac{f(n)}{g(n)} = 0\}
\end{align*}
In plain language, $\mathcal O(f(n))$ is the class of all functions which grow at most as fast as $f(n)$ (up to a multiplicative constant), $\Omega(f(n))$ is the class of all functions which grow at least as fast as $f(n)$, $\Theta(f(n))$ is the class of all functions which grow exactly as fast as $f(n)$, $o(f(n))$ is the class of all functions which grow slower than $f(n)$, and $\omega(f(n))$ is the class of all functions which grow faster than $f(n)$.

We sometimes abuse notation and write things like $\Omega(f(n)) \leq g(n)$, or $g(n) \leq \mathcal O(f(n))$, and so on. These simply mean $g(n)\in \Omega(f(n))$ and $g(n) \in \mathcal O(f(n))$ respectively.

\section{Upper bounds}\label{sec:upper}
There are three types of EPPA-witnesses which one should keep in mind:

\subsection{Finite homogeneous graphs}
A finite graph is \emph{homogeneous} if it is an EPPA-witness for itself. In 1976, Gardiner~\cite{Gardiner1976} classified all finite homogeneous graphs. Up to complementation, they are the following:
\begin{enumerate}
\item $C_5$, the cycle on five vertices,
\item $L(K_{3,3})$, the line graph of the complete bipartite graph with both partitions of size $3$, (this graph has nine vertices), and
\item disjoint unions of complete graphs of the same size.
\end{enumerate}
Clearly, the following holds:
\begin{observation}
A homogeneous graph is an EPPA-witness for all its subgraphs.
\end{observation}
While there are very few homogeneous graphs, one needs to keep them in mind, as their subgraphs tend to be exceptions in otherwise general lower bounds.

\begin{remark}\label{rem:homogeneous}
In general, a possibly infinite (model-theoretic) structure $\str M$ is \emph{homogeneous} if every isomorphism between finite substructures of $\str M$ extends to an automorphism of $\str M$. The Cherlin--Lachlan \emph{classification programme of homogeneous structures} (see e.g.~\cite{Lachlan1980,Lachlan1984,lachlan1984countable,Cherlin1998,Cherlin1999,Cherlin2013}) shows that homogeneous structures are very special; for example, Lachlan and Woodrow in 1980 proved that the only countably infinite homogeneous graphs are, up to isomorphism and complementation, the countable random graph (also called the Rado graph), its $K_k$-free analogues (called the Henson graphs), and disjoint unions of complete graphs of the same size.

Given a countable structure $\str M$, we can assume that its vertex set is the set $\mathbb N$ of natural numbers, and hence identify the automorphism group $\Aut(\str M)$ with a subgroup of the symmetric group $\Sym(\mathbb N)$. There is a natural topology of $\Sym(\mathbb N)$ (the \emph{pointwise convergence topology}) and it turns out that $\Gamma$ is a closed subgroup of $\Sym(\mathbb N)$ if and only if it is the automorphism group of some countable homogeneous structure.

It is thus, perhaps, not that surprising that various dynamical properties of closed subgroups of $\Sym(\mathbb N)$ are equivalent to combinatorial properties of classes of finite structures. EPPA is an example of this: If a hereditary class $\mathcal C$ of finite structures has EPPA and the \emph{joint embedding property} (for every $\str A,\str B\in \mathcal C$ there is $\str C\in \mathcal C$ containing both $\str A$ and $\str B$ as substructures) then $\mathcal C$ is the \emph{age} of some homogeneous structure $\str M$ (the class of all finite substructures of $\str M$, denoted by $\Age(\str M)$). Kechris and Rosendal in 2007 proved that if $\str M$ is locally finite then $\Aut(\str M)$ can be written as the closure of a union of a chain of compact subgroups if and only if $\Age(\str M)$ has EPPA.~\cite{Kechris2007}

Another example of such a property is the \emph{Ramsey property}. Kechris, Pestov, and Todor\v cevi\'c in 2005 proved that $\Aut(\str M)$ is extremely amenable if and only if $\Age(\str M)$ has the Ramsey property, that is, satisfies a structural variant of Ramsey's theorem.~\cite{Kechris2005} This can also be seen as a link between the study of EPPA numbers and the well-developed area of Ramsey numbers (see e.g.~\cite{Conlon2015}). See also~\cite{Geschke2011} for a cross-over result between EPPA and Ramsey.
\end{remark}

\subsection{Kneser graphs}
In 2000, Herwig and Lascar proved the following theorem, thereby answering Hrushovski's question:
\begin{theorem}[Herwig--Lascar, 2000~\cite{herwig2000}]\label{thm:hl}
If $G$ is a graph with $n$ vertices, $m$ edges and maximum degree $d\geq 2$ then $\eppa(G) \leq {dn - m \choose d}$. Consequently, $\eppa(n) \leq \left(\frac{3en}{4}\right)^n$, and graphs with maximum degree bounded by some constant $d$ have EPPA numbers at most $\mathcal O(n^d)$.
\end{theorem}
We give a sketch of the proof, for full details see Section~4.1 of~\cite{herwig2000}, which is fully self-contained and does not need any familiarity with the rest of the paper.
\begin{proof}[Sketch of a proof]
First assume that $G$ is $d$-regular and let $E$ be the set of edges of $G$. Construct a graph $H$ with vertex set $E \choose d$, the set of all $d$-element subsets of $E$, where $X$ is connected to $Y$ if and only if $X\cap Y\neq \emptyset$, that is, $H$ is isomorphic to the complement of the \emph{Kneser graph $K(\lvert E\rvert, d)$~\cite{Kneser1955}}. Define an embedding $\psi\colon G\to H$ such that $\psi(v) = \{e\in E : v\in e\}$. Note that every partial automorphism $\varphi$ of $G$ induces a partial permutation of $E$ which can be extended to a full permutation of $E$ such that its action on $E \choose d$ extends $\varphi$. Observing that every permutation of $E$ induces an automorphism of $H$ concludes the proof that $H$ is an EPPA-witness for $G$.

If $G$ is not $d$-regular, one can add ``half-edges'' to make it regular, or more precisely, to make all images in $\psi[G]$ have the same cardinality. Doing this in the optimal way leads to $\lvert H\rvert = {dn - m \choose d}$, and after observing that the complement of an EPPA-witness for a graph $G$ is an EPPA-witness for the complement of $G$, one can bound ${dn - m \choose d} \leq \left(\frac{3en}{4}\right)^n$ (see Section~4.1 of~\cite{herwig2000} for the computation).
\end{proof}
Note that the constructed EPPA-witness only depends on the parameters $n$, $m$ and $d$.

\subsection{Valuation graphs}
Herwig and Lascar's bound for $\eppa(n)$ is super-expo\-nential, while Hrushovski's lower bound is exponential. This gap has been closed by Evans, Hubička, Konečný and Nešetřil in 2018~\cite{eppatwographs} who proved that $\eppa(n) \in \mathcal O(n2^n)$. They actually proved a stronger variant of EPPA where one can extend all switching automorphisms, here we will use the slightly simplified presentation from Section~3 of~\cite{Hubicka2018EPPA}. We remark that the same construction has been independently discovered by Andr\'eka and N\'emeti~\cite{Andreka2019}.

\begin{theorem}[Evans--Hubička--Konečný--Nešetřil, 2018~\cite{eppatwographs}]\label{thm:hkn}
$$\eppa(n) \leq n2^{n-1}.$$
\end{theorem}
We give a sketch of the proof which appears in Section~3 of~\cite{Hubicka2018EPPA} (note that it is fully self-contained and does not require any familiarity with the rest of~\cite{Hubicka2018EPPA}).
\begin{proof}[Sketch of a proof from~\cite{Hubicka2018EPPA}]
Fix $n$. Define a graph $H_n$ whose vertices are all pairs $(i, f)$ where $1\leq i\leq n$ (it is called the \emph{projection}) and $f$ is a function $[n]\setminus\{i\} \to \{0,1\}$ where $[n] = \{1, \ldots, n\}$. We call $f$ a \emph{valuation function for $i$}. Connect $(i,f)$ and $(i',f')$ if and only if $i\neq i'$ and $f(i')\neq f'(i)$. This will serve as an EPPA-witness for all graphs on at most $n$ vertices.

Note that for every permutation $\pi$ of $[n]$ the function $\theta_\pi \colon H_n \to H_n$, defined by $\theta_\pi(i, f) = (\pi(i), f')$ where $f'(j) = f(\pi^{-1}(j))$, is an automorphism of $H_n$. Similarly, given $a<b\in [n]$, the function $\theta_{a,b}\colon H_n \to H_n$, defined by $\theta_{a,b}(i, f) = (i, f')$ where $f'(j) = 1-f(j)$ if $\{a,b\} = \{i,j\}$ and $f'(j) = f(j)$ otherwise, is also an automorphism of $H_n$.

Given a graph $G$ with $\lvert G\rvert \leq n$, define an embedding $\psi\colon G\to H_n$ by assuming that $G = [n]$ and sending $\psi(i) = (i, f_i)$ where $f_i(j) = 1$ if and only if $j<i$ and $ij$ is and edge of $G$. Now every partial automorphism $\varphi$ of $G$ induces a partial permutation of $[n]$ which one can extend to a permutation $\pi$ of $[n]$. It turns out that there is a unique choice of a set of pairs $\{a,b\} \subseteq [n]$ such that $\theta_\pi$ composed with all such $\theta_{a,b}$'s is an automorphism of $H_n$ extending $\varphi$.
\end{proof}

\section{Lower bounds}
In this section we will present some lower bounds for EPPA numbers. The first and the only published lower bound has already been proved in Hrushovski's paper. The \emph{half graph} on $2m$ vertices is the bipartite graph on the vertex set $[2m]$ where $i<j$ form and edge if and only if $j\geq i+m$.
\begin{theorem}[Hrushovski, 1992~\cite{hrushovski1992}]\label{thm:hrushovski_half}
Every EPPA-witness of the half-graph on $2m$ vertices has at least $2^m$ vertices.
\end{theorem}
Hrushovski shows that for any non-empty subset of one partition of the half graph, every EPPA-witness needs to contain a vertex which is connected to the subset and not connected to the rest of the partition. There are $2^m-1$ such subsets. The following lemma abstracts the key idea from Hrushovski's proof of Theorem~\ref{thm:hrushovski_half} and will be important for the majority of lower bounds in this section:

\begin{lemma}\label{lem:hrus}
Let $G$ be a graph on $n$ vertices containing an independent set $A$ of size $m$. Let $k_1 < \cdots < k_\ell$ be integers such that for every $1\leq i\leq \ell$ there exists $v_i\in G\setminus A$ which has exactly $k_i$ neighbours in $A$. Then $\eppa(G) \geq m + \sum_{i=1}^\ell {m\choose k_i}$.
\end{lemma}
\begin{proof}
Let $H$ be an EPPA-witness for $G$. We will prove that for every $X\subseteq A$ with $\lvert X\rvert = k_i$ for some $i$, there is a vertex $w\in H$ which is connected to every member of $X$ and no member of $A\setminus X$. This clearly implies the conclusion. Given such a set $X$, let $Y_i$ be the set of neighbours of $v_i$ inside $A$. By definition, $\lvert Y_i\rvert = \lvert X\rvert = k_i$, and thus there exists a permutation $f\colon A\to A$ such that $f(Y_i) = X$. As $f$ is a partial automorphism of $G$, it extends to an automorphism $g\colon H\to H$, and $g(v_i)$ is a vertex of $H$ connected to every member of $X$ and no member of $A\setminus X$.
\end{proof}
Note that complementation preserves EPPA numbers, hence an analogue of Lemma~\ref{lem:hrus} holds also for complete graphs instead of independent sets. More generally, if $G$ contains a subgraph $A$ with an orbit of injective $k$-tuples of size $m$, and a vertex which is connected precisely to members of one such tuple then the same argument gives $\eppa(G)\geq \frac{m}{k!}$

In this note, however, we will mostly be using a version of Lemma~\ref{lem:hrus} with only one $k_i$. For convenience, we state it as a corollary (with a slightly weaker conclusion).
\begin{corollary}\label{cor:more_general}
Let $G$ be a graph on $n$ vertices. Assume that it contains an independent set $A$ and a vertex $v$ which is connected to precisely $k$ vertices of $A$. Then $\eppa(G)\geq {\lvert A\rvert \choose k}$.\qed
\end{corollary}

This allows us to improve the lower bound on $\eppa(n)$:
\begin{proof}[Proof of Theorem~\ref{thm:lower_bound}]
We will construct a graph $G$ as follows: The vertex set of $G$ is $[n]$ and $G$ is bipartite with one partition $[n-1]$ and the other partition $\{n\}$, where $mn$ is an edge of $G$ if and only if $m\leq \frac{n-1}{2}$. Observe that $[n-1]$ forms an independent set and that vertex $n$ is connected to precisely $\lfloor \frac{n-1}{2}\rfloor$ vertices of the independent set, hence, by Corollary~\ref{cor:more_general}, $\eppa(n) \geq {n-1 \choose \lfloor \frac{n-1}{2}\rfloor}$. This is the central binomial coefficient with known asymptotics (see e.g.~\cite[Chapter 5.2]{MatousekProbabilistic}, it also follows easily from Stirling's approximation). 
\end{proof}
See Question~\ref{q:bad_graph} for more discussion of this example. Next, we focus on bounds parameterised by the maximum degree.

\begin{corollary}\label{cor:general}
Let $G$ be a graph on $n$ vertices with maximum degree $d$. Let $k$ be the maximum size of an independent set in the neighbourhood of a vertex of $G$. Then every EPPA-witness of $G$ has at least ${\lceil n/(d+1)\rceil \choose k}$ vertices.
\end{corollary}
\begin{proof}
Let $v$ be a vertex with an independent set $X$ of size $k$ in its neighbourhood. Create greedily an independent set $A$ of vertices of $G$ which contains $X$ (and, hence, no other neighbour of $v$). We have that $\lvert A\rvert\geq \lceil \frac{n}{d+1}\rceil$ and we can apply Corollary~\ref{cor:more_general}.
\end{proof}

\begin{corollary}\label{cor:bounded_quadratic}
Let $G$ be a graph on $n$ vertices with maximum degree $d$. Then either $G$ is a subgraph of a finite homogeneous graph, or the smallest EPPA-witness of $G$ has at least $\Omega((n/d)^2)$ vertices.
\end{corollary}
\begin{proof}
If the neighbourhood of every vertex forms a clique then $G$ is a disjoint union of complete graphs and hence a subgraph of a finite homogeneous graph. Otherwise the size of the largest independent set in the neighbourhood of some vertex is at least 2 and Corollary~\ref{cor:general} gives the desired conclusion.
\end{proof}

\begin{corollary}\label{cor:triangle-free}
Let $G$ be a triangle-free graph on $n$ vertices with maximum degree $d$. Then $\eppa(G)\geq {\lceil n/(d+1)\rceil \choose d}$. If $d$ is constant then $\eppa(G)\in \Theta(n^d)$.
\end{corollary}
\begin{proof}
If $G$ is triangle-free then the neighbourhoods of vertices contain no edges. Apply Corollary~\ref{cor:general} and use Theorem~\ref{thm:hl} for the upper bound.
\end{proof}

\begin{proof}[Proof of Theorem~\ref{thm:degrees}]
The first lower bound is given by Corollary~\ref{cor:bounded_quadratic} and the triangle-free case is handled by Corollary~\ref{cor:triangle-free}.
\end{proof}

Theorem~\ref{thm:degrees} implies, in particular, that cycles have quadratic EPPA numbers. The following corollary makes the bounds more precise:
\begin{corollary}\label{cor:cycles}
Let $C_n$ denote the cycle on $n$ vertices. For every $n$ we have $\eppa(C_n) \leq {n\choose 2}$. If $n$ is even then $\eppa(C_n) \geq \frac{n(n+2)}{8}$, and if $n$ is odd then $\eppa(C_n)\geq \frac{(n-1)(n+5)}{8}$.
\end{corollary}
Note that $C_3$, $C_4$ and $C_5$ are homogeneous and $C_6$ embeds into the homogeneous graph $L(K_{3,3})$. In a paper in preparation~\cite{BradleyEPPA}, we show that $\eppa(C_6) = \lvert L(K_{3,3})\rvert$ and $\eppa(C_7) = 21$, that is, the upper bound is tight for $C_7$.
\begin{proof}
The upper bound is given by Theorem~\ref{thm:hl}. The largest independent set in $C_n$ has size $\lfloor\frac{n}{2}\rfloor$ and (for $n\geq 4$) there are two vertices in the independent set which have a common neighbour (and this neighbour has no other edges). Moreover, if $n$ is odd, there is also a vertex which is connected to exactly one member of the independent set. Lemma~\ref{lem:hrus} thus gives the lower bounds.
\end{proof}
See also Problem~\ref{prob:cycles} and Conjecture~\ref{conj:cycles}. Next, we continue with a proof of Theorem~\ref{thm:bwc} (note that subgraphs of homogeneous graphs need to be excluded in the statement as they can have arbitrarily small $\eppa(G) - \lvert G\rvert$):
\begin{proof}[Proof of Theorem~\ref{thm:bwc}]
We say that a graph is $k$-homogeneous if any isomorphism between induced subgraphs on at most $k$ vertices extends to an automorphism. We use two group-theoretic ingredients in our proof (a permutation group $\Gamma$ on a set $X$ is \emph{transitive} if for every $x,y\in X$ there is $g\in \Gamma$ such that $gx = y$):
\begin{enumerate}
\item\label{item:neumann} (Neumann’s Separation Lemma~\cite{Neumann1975}). Let $\Gamma$ be a transitive permutation group on a set $X$ with $\lvert X\rvert = n$ and let $A,B\subseteq X$. If $\lvert A\rvert\lvert B\rvert < n$ then there exists $g\in \Gamma$ such that $gA\cap B = \emptyset$.
\item\label{item:cameron} (Cameron~\cite{Cameron1980}). A 5-homogeneous graph is homogeneous.
\end{enumerate}
Let $G$ be a graph on $n$ vertices with $\eppa(G) = m < \frac{5}{4}n$, and let $H$ be an EPPA-witness for $G$ with $m$ vertices. It is easy to see that $\Aut(H)$ is transitive (all vertices of $G$ need to lie in the same orbit, and if there were multiple orbits of vertices of $H$ then taking the subgraph induced on the orbit of vertices of $G$ would yield a smaller EPPA-witness). Since $\lvert H \setminus G\rvert < \frac{m}{5}$, if $C\subseteq H$ with $\lvert C\rvert \leq 5$, by~(\ref{item:neumann}) there exists $g\in \Aut(H)$ such that $gC \cap (H\setminus G) = \emptyset$, that is $gC \subseteq G$.

Let $A, B$ be subsets of vertices of $H$ with $\lvert A\rvert = \lvert B\rvert \leq 5$ and let $f\colon A\to B$ be a partial automorphism. By the previous paragraph there are $h_A, h_B\in \Aut(H)$ such that $h_A(A) \subseteq G$ and $h_B(B) \subseteq G$. Hence, $h_Bfh_A^{-1}$ is a partial automorphism of $G$ which extends to an automorphism $g\in \Aut(H)$. Consequently, $h_B^{-1}gh_A$ is an automorphism of $H$ extending $f$, that is, $H$ is 5-homogeneous, and so it is homogeneous by~(\ref{item:cameron}).
\end{proof}
We conjecture that a quadratic lower bound is true (see Conjecture~\ref{conj:quadratic}). Note that the same argument shows that if $G$ has an EPPA-witness on fewer than $\frac{t}{t-1}\lvert G\rvert$ vertices then it is a subgraph of a $t$-homogeneous graph. Using the Classification of Finite Simple Groups, these graphs are all known. In particular, there are only finitely many finite $4$-homogeneous graphs which are not homogeneous~\cite{Buczak,Cameron1985}.

Theorem~\ref{thm:bwc} gives a lower bound for all graphs except for subgraphs of the homogeneous ones. Theorem~\ref{thm:random} says that, in the probabilistic sense, almost all graphs have at least {\sl almost} quadratic EPPA numbers.

\begin{proof}[Proof of Theorem~\ref{thm:random}]
Fix some $\varepsilon > 0$ which will be determined later. Asymptotically almost surely, $G(n,1/2)$ contains an independent set $I$ of size $(2-\varepsilon)\log_2(n)$ (this is a standard result in random graphs proved originally by Matula~\cite[Theorem 6]{Matula1976}). 
Again, asymptotically almost surely, there is a vertex outside of $I$ connected to precisely $k$ members of $I$ with $(1-\varepsilon)\log_2(n) \leq k \leq \log_2(n)$ (this follows by basic computation, see e.g.~\cite[Section 7]{MatousekProbabilistic}).
Corollary~\ref{cor:more_general} gives $$\eppa(G(n,1/2))\geq {(2-\varepsilon)\log_2(n) \choose k}.$$ Monotonicity and symmetry of binomial coefficients imply that
$$\eppa(G(n,1/2))\geq {(2-\varepsilon)\log_2(n) \choose k} \geq {(2-\varepsilon)\log_2(n) \choose (1-\varepsilon)\log_2(n)} \geq {2(1-\varepsilon)\log_2(n) \choose (1-\varepsilon)\log_2(n)}.$$
The standard bound on the central binomial coefficient (see e.g.~\cite[Chapter 5.2]{MatousekProbabilistic}) implies, for some constant $c>0$, $$\eppa(G(n,1/2)) \geq c\frac{4^{(1-\varepsilon)\log_2(n)}}{\sqrt{(1-\varepsilon)\log_2(n)}} \in \Omega(n^{2-2\varepsilon} / \sqrt{\log(n)}),$$ hence picking $\varepsilon < \frac{\delta}{2}$ gives the desired result.
\end{proof}
As one instance of Problem~\ref{prob:random} we propose the problem of improving these bounds. Theorem~\ref{thm:hl} implies that bounded degree graphs have polynomial EPPA numbers. On the other hand, Theorem~\ref{thm:lower_bound} presents graphs of average degree 1 which have exponential EPPA numbers. We conclude this section with a proof of Theorem~\ref{thm:random_bddegree} which says that, in fact, almost all graphs of constant average degree have superpolynomial EPPA numbers.

\begin{proof}[Proof of Theorem~\ref{thm:random_bddegree}]
Fix $c$ and $d$. It is well-known that there exists a constant $c'$ such that, asymptotically almost surely, $G(n,c/n)$ has an independent set of size at least $c'n$. (A constant fraction of the vertices induce a forest and, for example, all leaves form a linear-size independent set, see e.g.~\cite[Chapter 11]{Alon2004}.) Hence we can assume that $A$ is such an independent set. Let $v$ be an arbitrary vertex not in $A$ and let $p_d$ be the probability that $v$ has exactly $d$ neighbours in $A$. We have
\begin{align*}
p_d &= {c'n \choose d}\left(c/n\right)^d \left(1-c/n\right)^{c'n-d}\\
&\geq \left(c'n/d\right)^d\left(c/n\right)^d e^{-2\frac{c}{n}(c'n-d)},
\end{align*}
where we used standard bounds ${n\choose k} \geq (\frac{n}{k})^k$, and $1-p \geq e^{-2p}$ for $0\leq p\leq \frac{1}{2}$. Note that $\left(c'n/d\right)^d\left(c/n\right)^d$ is a positive constant. Similarly, $c'n-d \in \Theta(n)$, and so $e^{-2\frac{c}{n}(c'n-d)} \in \Theta(1)$. Consequently, $p_d\in \Theta(1)$ and hence, asymptotically almost surely, there will be a vertex $v$ in $\eppa(G(n,c/n))$ which is connected to exactly $d$ members of $A$, which allows us  to apply Corollary~\ref{cor:more_general}.
\end{proof}
We remark that a result of Frieze~\cite{Frieze1990} should make it possible to extend Theorem~\ref{thm:random_bddegree} to graphs $G(n,p(n))$ for functions $p(n)$ which are asymptotically slightly larger than $c/n$.

\section{Further directions and questions}
EPPA-witnesses are very symmetric objects. First, it is easy to see that smallest EPPA-witnesses are vertex transitive, and smallest EPPA-witnesses which moreover minimize the number of edges are also edge-transitive. Second, if $H$ is an EPPA-witness for $G$ then orbits of tuples of vertices $G$ in $H$ are determined by the isomorphism type of the subgraph induced on the tuple. Thus, in particular, every vertex of $G$ has the same degree in $H$, every pair of vertices connected by an edge has the same cardinality of the intersection of their neighbourhoods, and similarly for non-edges (cf. \emph{strongly regular graphs}, see e.g.~\cite{Godsil2001}), and so on.

Thus, the EPPA number can be seen as a measure of asymmetry of a graph -- a very asymmetric graph should be hard to symmetrise and hence have large EPPA number (although with the current results it is possible that in reality EPPA numbers only measure something much simpler such as the maximum degree). Besides this and it being a nice problem with elegant results, we believe that the study of EPPA numbers will lead to an improved understanding of structural properties which make a graph easy or hard for EPPA, as well as to new methods for proving upper bounds. In turn, these new methods might help solve some important open problems in the wider area of EPPA for structures, such as the questions presented in the following subsection:

\subsection{Existence of EPPA-witnesses for classes of structures}
\begin{question}[Herwig--Lascar, 2000~\cite{herwig2000}]\label{q:hl}
Does the class of all finite tournaments have EPPA? Or explicitly: Is it true that for every finite tournament $A$ there is a finite tournament $B$ containing $A$ as a sub-tournament such that every isomorphism of sub-tournaments of $A$ extends to an automorphism of $B$?
\end{question}
This is, perhaps, the biggest open problem of the area. Herwig and Lascar show that it is equivalent to a statement about the \emph{profinite topology} (or, more precisely, \emph{odd-adic topology}) on free groups~\cite{herwig2000}. One of the reasons why this problem has, so far, eluded all attempts at solving it, is that all automorphisms of tournaments have odd order, while the existing methods for proving EPPA for general structures (see e.g.~\cite{Hubicka2018EPPA} or~\cite{Konecny2023phd}) rely heavily on automorphisms of even orders. We remark that Huang, Pawliuk, Sabok, and Wise~\cite{Sabok} proved that a custom-tailored class of hypertournaments does not have EPPA by showing a failure of an equivalent profinite topology statement.

Hubička, Jahel, Konečný, and Sabok~\cite{HubickaSemigeneric} suggest a weakening of Question~\ref{q:hl} which allows EPPA-witnesses to have some even-order automorphisms, but which still seems to resist the existing methods.
\begin{question}[Hubička--Jahel--Konečný--Sabok, 2023~\cite{HubickaSemigeneric}]
Is there a natural number $\ell$ such that for every finite tournament $A$ there is a finite directed graph $B$ satisfying the following:
\begin{enumerate}
	\item $B$ contains $A$ as a subgraph,
	\item $B$ contains no bi-directional edges,
	\item\label{q:tour:3} $B$ contains no independent sets of size $\ell$, and
	\item every isomorphism of subgraphs of $A$ extends to an automorphism of $B$?
\end{enumerate}
\end{question}
It is easy to obtain such directed graphs $B$ without condition~(\ref{q:tour:3}), see Problem~\ref{prob:directed}. Question~\ref{q:hl} is precisely this question for $\ell=2$.

In the spirit of this paper, one can ask for bounds on EPPA numbers for (small) tournaments. There are two finite homogeneous tournaments, the 1-vertex tournament, and the oriented triangle~\cite{lachlan1984countable}. At present, there is only one other tournament EPPA-witness known to us, namely the seven-vertex Payley tournament which is an EPPA-witness for the transitive triangle.\footnote{The fourth author learned about the Payley tournament EPPA-witness from Sofia Brenner.}
It would thus be interesting to see some non-trivial lower bounds for transitive tournaments on, say, three or four vertices. As a heuristic for judging how surprising a lower bound for a concrete tournament is, one would expect that constructing EPPA-witnesses for tournaments to be at least as difficult as constructing triangle-free EPPA-witnesses for graphs (in both cases one has a single binary relation and a condition on triples of vertices), see Section~\ref{sec:kkfree}.

\begin{problem}
Prove bounds for the EPPA numbers of some (small) tournaments.
\end{problem}
We hope that working on this problem will lead to some observations useful even in the context of Question~\ref{q:hl}.

\medskip

Besides tournaments and directed graphs with no large independent sets, there are also other interesting classes of structures for which EPPA is open:
\begin{question}[Hubička--Konečný--Nešetřil, 2019~\cite{Hubicka2018EPPA}]
Does the class of all finite partial Steiner triple systems have EPPA for closed substructures?
\end{question}
A partial Steiner triple system is a 3-uniform hypergraph where every pair of vertices is together in at most one hyperedge. In these structures, a pair of vertices in a hyperedge defines a unique third vertex, thereby giving rise to binary closures. It is then easy to construct counterexamples to EPPA for the standard hypergraph embeddings. However, the natural morphisms for this class respect closures, that is, a sub-hypergraph $A\subseteq B$ is a Steiner sub-system if and only if for every $u,v\in A$, if there is $w\in B$ such that $uvw$ form a triple then $w\in A$ (these are called \emph{strongly induced subsystems} by Bhat, Nešetřil, Reiher, and R\"odl~\cite{bhat2016ramsey}). If we only consider such closed substructures and partial automorphisms between them, the question remains open.

\begin{question}[Evans--Hubička--Konečný--Nešetřil, 2019~\cite{eppatwographs}]
Does the class of all finite graphs with complementing embeddings have EPPA?
\end{question}
Given graphs $G,H$, a map $f\colon G\to H$ is a \emph{complementing embedding} if either $f$ is an embedding $G\to H$, or $f$ is an embedding $G\to \bar{H}$ where $\bar{H}$ is the complement of $H$. In~\cite{eppatwographs}, Evans, Hubička, Konečný, and Nešetřil prove EPPA for graphs with \emph{switching embeddings} (where one can pick a subset of vertices and switch edges with one endpoint in this subset to non-edges and vice versa), which is one of the five so-called \emph{reducts} of the countable random graph (corresponding to two-graphs), see e.g.~\cite{Thomas1991}. Graphs with complementing embeddings are another reduct of the random graph. In his Bachelor thesis, Beliayeu~\cite{Beliayeu2023bc} proved EPPA for the class of all finite graphs with loops with complementing embeddings (that is, every vertex may, but does not have to, have a loop, and loops also get complemented). This created a particularly interesting situation: It is open whether graphs with complementing embeddings have EPPA, but there are currently two \emph{expansions} of graphs with complementing embeddings that are known to have EPPA: Graphs with normal embeddings, and graphs with loops with complementing embeddings. Moreover, both of them can be seen as reducts of countable random graphs with a generic unary mark (i.e., unary relation), and they are incomparable. It would be interesting to see if the meet of these two expansions are graphs with complementing embeddings, and if not then whether the meet has EPPA.

\subsection{Bounds on EPPA numbers for graphs}
Besides these questions, there are of course many open problems about EPPA numbers for graphs, let us pose a few of them:

\begin{question}
There is still a factor $n^{3/2}$ difference between the lower bound and the upper bound for $\eppa(n)$. What is the correct bound?
\end{question}
We conjecture that $\Omega(\frac{2^n}{\sqrt{n}})$ is not a tight lower bound. The reason for this conjecture is that, in Theorem~\ref{thm:lower_bound}, we only needed partial automorphisms which do not have vertex $n$ in the domain. This leaves a plethora of other automorphisms to use to improve the bound. In particular, we can send vertex $n$ to any vertex from $[n-1]$ and then, possibly, arbitrarily permute vertices from $[n-1]$ which are not connected to $n$ and not the image of $n$. It seems that these partial automorphisms force EPPA-witnesses to be much more complicated. This motivates the following question:
\begin{question}\label{q:bad_graph}
Let $G$ be the graph from Theorem~\ref{thm:lower_bound}, that is, the $(\lfloor n/2\rfloor+1)$-vertex star together with $n-(\lfloor n/2\rfloor+1)$ isolated vertices. What is the EPPA number of $G$? Can one at least improve one of the bounds $\Omega(2^n/\sqrt{n}) \leq \eppa(G) \leq n2^{n-1}$?
\end{question}

We intentionally ask for an asymptotic improvement of the lower bound and an arbitrary improvement of the upper bound. The reason is that while one can likely improve the lower bound slightly by making a more precise computation, or possibly by showing that in order to extend a few partial automorphism with the high-degree vertex in the domain, the EPPA-witnesses need to be a bit larger, if one denotes by $\Gamma$ the group generated by the automorphisms of $H_n$ (see Theorem~\ref{thm:hkn}) needed to extend partial automorphisms of $G$ then all vertices of $H_n$ lie in the orbit of vertices of $G$ under the action of $\Gamma$. Simply said, in order to improve the upper bound even by a constant, one will likely need to make at least some changes to the construction and we do not see any obvious way of doing so.

There are other concrete graphs classes which we believe are good candidates to look at:
\begin{problem}\label{prob:cycles}
Corollary~\ref{cor:cycles} says that $\frac{1}{8}n^2 + o(n^2) \leq \eppa(C_n) \leq \frac{1}{2}n^2 + o(n^2)$ where $C_n$ is the cycle on $n$ vertices. Find the correct constant for the quadratic term.
\end{problem}
For the same reason as above we expect that the lower bound is not tight. In fact, we conjecture that the upper bound is tight:
\begin{conjecture}\label{conj:cycles}
If $n\geq 7$ then $\eppa(C_n) = {n\choose 2}$.
\end{conjecture}

\begin{problem}
Prove lower bounds on EPPA numbers for planar graphs.
\end{problem}
Planar graphs seem to be close to the graphs for which we were able to prove non-trivial lower bounds in that they have constant average degree and moreover they are 4-colourable, hence contain large independent sets, which gives hope that one could adapt the methods from this paper to show a lower bound. On the other hand, note that the star on $n$ vertices is a planar graph with a large independent set and a large maximum degree, but it is a subgraph of a homogeneous graph (the complement of the disjoint union of two copies of $K_{n-1}$) and thus has a linear EPPA number, so statements of the form ``either $G$ is a subgraph of a homogeneous graph, of $\eppa(G)$ is at least something'' are what one needs to aim for. As a particular instance of this problem, can one prove a quadratic lower bound for all planar graphs, not only the bounded-degree ones, hence extending Corollary~\ref{cor:bounded_quadratic}? See also Conjecture~\ref{conj:quadratic}.

Note that the graph from Theorem~\ref{thm:lower_bound} is planar and one can get it from the star (with a linear EPPA number) by adding isolated vertices. This also shows that the EPPA number can grow significantly by adding just a single isolated vertex.

\begin{question}
What is the asymptotically slowest growing function $f$ such that for every graph $G$ and every $v \in G$ which has no neighbours in $G$ we have that $\eppa(G)\leq f(\eppa(G-v))$? 
\end{question}
Theorem~\ref{thm:hkn} implies that $f(n) = (n+1)2^n$ works. On the other hand, if $G$ is the star on $n$ vertices together with an isolated vertex $v$ then, as we have seen above, $\eppa(G-v) = 2(n-1)$, while $\eppa(G)$ is at least quadratic by Corollary~\ref{cor:more_general}. Thus $f(n)$ is at least quadratic. Note that if we do not require that $v$ is isolated then one can remove the high-degree vertex from the construction in Theorem~\ref{thm:lower_bound} and get a jump from a linear EPPA number to $\Omega(2^n/\sqrt{n})$. Bradley-Williams and Cameron~\cite{BradleyEPPA} show that line graphs $L(K_{n,n})$ and slight modifications of Kneser graphs behave nicely with respect to adding isolated vertices when serving as EPPA-witnesses.

\begin{problem}\label{prob:random}
Improve upper and/or lower bounds for $\eppa(G(n, f(n)))$, the random graph model where every edge is present independently with probability $f(n)$ where $f(n)$ is some reasonable function. More formally, prove that $$\lim_{n\to \infty} \mathbb P[\eppa(G(n, f(n))) \leq b(n)] = 1$$ for some function $b(n)$ (or $\geq b(n)$ for a lower bound).
\end{problem}
In this note we give two partial solutions to this problem. For $f(n) = \frac{1}{2}$ we prove an almost-quadratic lower bound (see Theorem~\ref{thm:random}), and for $f(n) = \frac{c}{n}$ we prove that the EPPA numbers are superpolynomial (see Theorem~\ref{thm:random_bddegree}). We conjecture that the almost-quadratic bound for $G(n,1/2)$ is far from optimal. The reason is that we believe that Lemma~\ref{lem:hrus} is suitable for sparse graphs but it is much less useful for dense graphs. Another possible ``local'' approach for $G(n,1/2)$ is showing that any graph which ``regularises'' $G(n,1/2)$ as a subgraph needs to be large, where by regularising we mean that it contains $G(n,1/2)$ as a subgraph and all degrees of vertices of $G(n,1/2)$ are equal, all sizes of the intersections of the neighbourhoods of the endpoints of edges are equal etc. While it only takes roughly $\sqrt{n}$ extra vertices to make all degrees equal, perhaps it is much harder to do so for pairs, triples, \ldots

\begin{problem}
Study EPPA numbers of Paley graphs.
\end{problem}
Paley graphs are quasirandom of density $\frac{1}{2}$ hence they are, in a way, similar to $G(n,1/2)$. On the other hand, they are self-complementary and very symmetric (for example, both $C_5$ and $L(K_{3,3})$ are Paley graphs). As such, they might be a possible answer to the following meta-problem:

\begin{problem}
Find graph classes whose EPPA numbers have interesting behaviour.
\end{problem}

\medskip

Theorem~\ref{thm:degrees} implies that bounded degree graphs have at least quadratic EPPA numbers, but we believe that it holds in general, that one can improve the bound in Theorem~\ref{thm:bwc} to quadratic.
\begin{conjecture}\label{conj:quadratic}
A graph $G$ on $n$ vertices is either a subgraph of a finite homogeneous graph, or every EPPA-witness of $G$ has at least $\Omega(n^2)$ vertices.
\end{conjecture}
In order to disprove this conjecture one would need to find a new construction of EPPA-witnesses for some graphs and such a result would be very interesting.

Even smaller improvements of Theorem~\ref{thm:bwc} are interesting. For example, Theorem~\ref{thm:bwc} implies that graphs with maximum degree $o(\sqrt{n})$ have super-linear EPPA-witnesses. Can one prove this also for graphs with maximum degree equal to $\sqrt{n}$?

\medskip

So far, the only known polynomial upper bounds on EPPA numbers hold for graphs with bounded degrees and their complements, and for subgraphs of homogeneous graphs. This motivates the following question:

\begin{question}
Is there a class $\mathcal C$ of finite graphs satisfying the following:
\begin{enumerate}
\item For every $d$ there is $G\in \mathcal C$ with a vertex of degree at least $d$ and also a vertex of degree at most $\lvert G\rvert-d$,
\item no finite homogeneous graph is an EPPA-witness for any member of $\mathcal C$, and
\item there is a polynomial $f(n)$ such that for every $G\in \mathcal C$ it holds that $\eppa(G) \leq f(\lvert G\rvert)$.
\end{enumerate}
\end{question}

\subsection{Coherent EPPA numbers}
There is a strengthening of EPPA called \emph{coherent EPPA} which was introduced by Siniora and Solecki~\cite{Siniora} in order to obtain stronger dynamical properties of automorphism groups (see Remark~\ref{rem:homogeneous}). A structure $\str B$ is a \emph{coherent EPPA-witness} for a structure $\str A$ if it is an EPPA-witness for $\str A$ and there is a map $\Psi$ from partial automorphisms of $\str A$ to automorphisms of $\str B$ such that $\Psi(f)$ extends $f$, and if $\str C, \str D, \str E\subseteq \str A$ are substructures of $\str A$, and $f\colon \str C\to \str D$ and $g\colon \str D\to \str E$ partial automorphisms of $\str A$ then $\Psi(gf) = \Psi(g)\Psi(f)$.

\begin{lemma}
$\str B$ is a coherent EPPA-witness for $\str A$ if and only if it is an EPPA-witness for $\str A$ and there is a map $\Psi$ from partial automorphisms of $\str A$ to automorphisms of $\str B$ such that for every $\str C\subseteq \str A$ and automorphisms $f,g\colon \str C\to \str C$ we have that $\Psi(gf) = \Psi(g)\Psi(f)$.
\end{lemma}
In other words, in the definition we can restrict to looking at automorphisms of substructures of $\str A$.
\begin{proof}
One direction is clear. For the other direction, pick a representative of every isomorphism type of a subgraph of $\str A$. Given $\str C\subseteq \str A$, denote by $r(\str C)\subseteq \str A$ this representative. Moreover, for every $\str C\subseteq \str A$, pick one isomorphism $\iota_\str C\colon r(\str C)\to \str C$ such that if $\str C=r(\str C)$ then $\iota_\str C$ is the identity, and an arbitrary extension $\Psi(\iota_\str C)$ to an automorphism of $\str B$ (where all identities are extended to the identity automorphism). Put $\Psi(\iota_\str C^{-1}) = \Psi(\iota_\str C)^{-1}$. 

Every partial automorphism $f\colon \str C\to \str D$ can be written as $f = \iota_\str D\alpha \iota_\str C^{-1}$ where $\alpha$ is some automorphism of $r(\str C)=r(\str D)$ which is determined uniquely by $f$. Define $\Psi(f) = \Psi(\iota_\str D)\Psi(\alpha)\Psi(\iota_\str C)^{-1}$.

Note that if $g\colon \str D\to \str E$ is a partial automorphism with $g = i_\str D\beta \iota_\str E^{-1}$ then $gf = i_\str E\beta\alpha \iota_\str C^{-1}$, and consequently
\begin{align*}
	\Psi(gf) &= \Psi(i_\str E)\Psi(\beta\alpha)\Psi(\iota_\str C)^{-1}\\
			&= \Psi(i_\str E)\Psi(\beta)\Psi(\alpha)\Psi(\iota_\str C)^{-1}\\
			&= \Psi(i_\str E)\Psi(\beta)\Psi(\iota_\str D)^{-1}\Psi(\iota_\str D)\Psi(\alpha)\Psi(\iota_\str C)^{-1}\\
			&= \Psi(g)\Psi(f).
\end{align*}
\end{proof}

Almost all classes for which EPPA is known are also known to have coherent EPPA. A currently remaining exception to this are the class of all finite two-graphs for which this issue has been highlighted since 2018~\cite{eppatwographs}. There are two recent additions to the ``naughty list'': $n$-partite tournaments and semigeneric tournaments~\cite{HubickaSemigeneric}. Both Kneser graphs and valuation graphs from Section~\ref{sec:upper} are in fact coherent EPPA-witnesses if one makes sure to do things in a canonical way when extending partial automorphisms. Let $\ceppa(G)$ be the number of vertices of the smallest coherent EPPA-witness for $G$. It would be interesting to see if one can show some separation between $\eppa(G)$ and $\ceppa(G)$:
\begin{question}\label{q:ceppa}
Is there a graph $G$ for which $\eppa(G) \neq \ceppa(G)$?
\end{question}

If the answer to the above question is positive then it is natural to ask whether it is possible for coherent EPPA numbers to grow asymptotically faster than EPPA numbers.
\begin{question}
Is there a class of finite graphs $\mathcal C$ such that $\max\{\ceppa(G) - \eppa(G) : G\in \mathcal C, \lvert G\rvert = n\} \in \omega(1)$.
\end{question}

Note that we always have $\eppa(G)\leq \ceppa(G)$, and the upper bounds from Section~\ref{sec:upper} hold also for coherent EPPA numbers.

Motivated by a recent paper of Sági~\cite{Sagi2024} (in particular, Remark 4.3), one can introduce the following interpolation of EPPA and coherent EPPA: Let $\str A$ be a structure and let $\str B$ be an EPPA-witness for $\str A$. We say that $\str B$ is \emph{Aut-coherent} if there is a group embedding $\psi\colon \Aut(\str A) \to \Aut(\str B)$ such that $\psi(g)$ extends $g$ for every $g\in \Aut(\str A)$. Or in other words, we only want that full automorphisms of $\str A$ get extended coherently.

\begin{question}
Is there a graph $G$ such that its smallest Aut-coherent EPPA-witness has more vertices that $\eppa(G)$?
\end{question}
A positive answer to this question would imply a positive answer to Question~\ref{q:ceppa}, and it might be easier to work with.

\subsection{$K_k$-free EPPA numbers}\label{sec:kkfree}
In 1995, Herwig proved that if $G$ is a triangle-free graph then there is a triangle-free graph $H$ which is an EPPA-witness for $G$~\cite{Herwig1995}, and subsequently he proved this for $K_k$-free graphs for all $k\geq 3$~\cite{herwig1998}. This is in sharp contrast with the constructions from Section~\ref{sec:upper} which contain many triangles. It would be interesting to understand how much harder it is to produce $K_k$-free EPPA witnesses:
\begin{problem}
Prove upper- and lower- bounds on the size of $K_k$-free EPPA-witnesses for $K_k$-free graphs.
\end{problem}

Herwig's proofs do not immediately yield a bound on the size of the constructed EPPA-witnesses. He, however, provides a $2^{2^{2^{\mathcal O(n)}}}$ bound for an auxiliary lemma. Besides Herwig's constructions for $K_k$-free EPPA-witnesses, a different construction has been discovered by Hodkinson and Otto~\cite{hodkinson2003}. Their idea has been refined by Evans, Hubička, and Nešetřil~\cite{Evans3} and then further refined by Hubička, Konečný and Nešetřil (see Section~7 of~\cite{Hubicka2018EPPA}). The version presented in Section~7 of~\cite{Hubicka2018EPPA} gives, combined with Theorem~\ref{thm:hkn}, a rough upper bound of $2^{2^{2^{\mathcal O(n)}}}$, the proofs in~\cite{hodkinson2003} and~\cite{Evans3} are in principle the same as the proof from~\cite{Hubicka2018EPPA} and hence lead to the same upper bounds. 

Nevertheless, it is possible to improve the bounds that the Hodkinson--Otto-type constructions give for $K_k$-free graphs (or, in general, for relational free amalgamation classes given by finitely many forbidden irreducible substructures), and get an $2^{2^{\mathcal O(kn)}}$ upper bound. The construction is as follows:

Let $G$ be a finite $K_k$-free graph and let $H_0$ be an EPPA-witness for $G$ (not necessarily $K_k$-free). Given a vertex $u\in H_0$, let $U(u)$ be the set of all subgraphs of $H_0$ containing $u$ which are isomorphic to $K_k$. We say that a function $\chi\colon U(u) \to \{1,\ldots, k-1\}$ is a \emph{valuation function for $v$}. A pair $(u,\chi)$, $(v,\xi)$, where $\chi$ is a valuation function for $u$ and $\xi$ is a valuation function for $v$, is \emph{generic} if $u\neq v$ and for every $A \in U(u)\cap U(v)$ it holds that $\chi(A)\neq \xi(A)$.

We define a graph $H$ whose vertices are all pairs $(u,\chi)$ where $u\in H_0$ and $\chi$ is a valuation function for $u$, such that $(u,\chi)$ and $(v,\xi)$ form and edge of $H$ if and only if the pair is generic and $u$ and $v$ form an edge of $H_0$. Clearly, $H$ is $K_k$-free: If there was a copy $\widetilde{A}$ of $K_k$ in $H$ then, by looking at the first coordinate, we get a corresponding copy $A$ of $K_k$ in $H_0$. Thus $A$ is in the domain of all valuation functions of vertices of $\widetilde{A}$, and they all give it pairwise different valuations, which is a contradiction as there are only $k-1$ possible different valuations. To get a copy of $G$ in $H$, one can fix an injective function $A\cap G\to \{1,\ldots,k-1\}$ for every copy $A$ of $K_k$ in $H_0$, and define valuation functions for vertices of $G$ according to these functions. The proof that $H$ is an EPPA-witness for $G$ is identical to the proofs in~\cite{hodkinson2003} or in Section~7 of~\cite{Hubicka2018EPPA} (it is, in fact, a coherent EPPA witness).

Note that if $H_0$ has $m$ vertices then $H$ has at most $m(k-1)^{m - 1\choose {k-1}}$ vertices. If we use Theorem~\ref{thm:hkn} to produce $H_0$, we have $\lvert H_0\rvert = n2^{n-1}$, and thus $\lvert H\rvert \in 2^{2^{\mathcal O(kn)}}$.

Even though the best upper bound is doubly exponential, the best lower bound known to the authors is $\Omega(2^n/\sqrt{n})$ from Theorem~\ref{thm:lower_bound}. Sabok asked (in private communication, see Question~2.0.9 in~\cite{Konecny2023phd}) if there is a class with smallest EPPA-witnesses larger than $2^{n^d}$ for every $d$. $K_k$-free graphs are a candidate for a positive answer to this question.
Another candidate are semigeneric tournaments for which a $2^{2^{\mathcal O(n)}}$ upper bound has been proved recently by Hubička, Jahel, Konečný, and Sabok~\cite{HubickaSemigeneric}.

\subsection{EPPA numbers of directed graphs}
\begin{problem}\label{prob:directed}
Study EPPA numbers of directed graphs.
\end{problem}
To our best knowledge there are no explicitly stated results for directed graphs. In Section~4.2 of~\cite{herwig2000}, Herwig and Lascar generalise their Kneser graph construction to arbitrary relational structures as follows. For simplicity, assume that we have only one $r$-ary relation (otherwise one can take a product). Consider a finite set $X$ and define a structure whose vertices are all $r$-tuples of $d$-element subsets of $X$. Tuples $(a_1^1, a_2^1,\ldots,a_r^1),\ldots,(a_1^r, a_2^r,\ldots,a_r^r)$ are in the relation if and only if $\bigcap_{i=1}^r a_i^i \neq \emptyset$. For directed graphs with maximum degree $d$ this gives an $\mathcal O(n^{2d})$ upper bound.

In Section~4 of~\cite{Hubicka2018EPPA}, Hubička, Konečný, and Nešetřil also give a valuation function construction for arbitrary relational structures. With one injective relation of arity $r$ this amounts to having valuation functions from all injective $r$-tuples to $\{0,1\}$, and putting a tuple into a relation if and only if no two elements have the same projection and the sum of the valuation functions evaluated at the projections is odd, thereby giving an $\mathcal O(2^{n^r})$ upper bound, which specializes to $\mathcal O(2^{n^2})$ for directed graphs.

However, one can make a more efficient valuation function construction by letting valuation functions for vertex $i$ be functions from $[n]\setminus\{i\}$ to $\mathbb Z_4$, and deciding one of the four possibilities for the pair $(i,f)$ and $(j,g)$ (that is, no edges, both edges or one edge in one of the two directions) according to the value of $f(j) + g(i)$. If we disallow bi-directional edges then in fact $\mathbb Z_3$ is enough as the domain. This leads to upper bounds of $n4^{n-1}$ and $n3^{n-1}$ respectively. We remark that a similar construction has been used by Beliayeu~\cite{Beliayeu2023bc} to prove EPPA for complementing directed graphs with loops.

Many of the lower bounds presented here generalise immediately to directed graphs. A slight modification to the construction from Theorem~\ref{thm:lower_bound} (have one direction of an edge to $\frac{n}{3}$ vertices, another to another $\frac{n}{3}$ vertices and no edge to the remaining $\frac{n}{3}$ vertices) gives a lower bound of $\Omega(\frac{3^n}{n})$ for directed graphs with no bi-directional edges, and an analogous modification gives $\Omega(\frac{4^n}{n^{3/2}})$ for directed graphs with bi-directional edges. It would, however, be interesting to see, for example, what the EPPA numbers are for oriented cycles, and whether they differ from cycles where each edge is oriented randomly. Or what the EPPA numbers are for linear orders (seen as transitive tournaments).

\subsection{EPPA numbers of hypergraphs}
\begin{problem}
Study EPPA numbers of hypergraphs.
\end{problem}
As we mentioned above, both the Kneser graph and the valuation function constructions generalise to hypergraphs. For the Kneser graph, regularisation is technically more challenging and leads to an upper bound of $2^{r!rn^r}$ for $r$-uniform hypergraphs (see Section~4.2 of~\cite{herwig2000}), but should, again, give better bounds for bounded degree hypergraphs for a suitable notion of a degree. With valuation functions, one can have the domain of valuation functions for vertex $i$ be the set of all $(r-1)$-elements subsets of $[n]\setminus\{i\}$, leading to an upper bound of $n 2^{n-1 \choose {r-1}}$ (see Remark~3.6 of~\cite{Hubicka2018EPPA}).

Contrary to directed graphs, the lower bound construction from Theorem~\ref{thm:lower_bound} does not adapt to give a matching lower bound for hypergraphs. We can, however, get at least a superexponential lower bound:
\begin{observation}\label{obs:3uniform}
For every $m = 2^k$, there is a 3-uniform hypergraph $G$ on $n=m+k+1$ vertices such that every EPPA-witness for $G$ has at least $m! \in 2^{\Omega(n\log(n))}$ vertices.
\end{observation}
\begin{proof}
The vertices of $G$ will be partitioned into three parts $A,B,C$ with $A = \{a\}$, $B = \{b_0,\ldots,b_{k-1}\}$ and $C = \{0,\ldots,m-1\}$. We put $\{x,y,z\}$ to be a hyperedge of $G$ if and only if $x=a$, $y=b_i$ for some $0\leq i< k$ and $z\in C$ is such that it has 1 on the $i$-th position when written in binary. There are no other hyperedges in $G$.

Let $H$ be an EPPA-witness for $G$. Given $v\in H$, we define a function $f_v\colon C\to C$ such that the $i$-th binary bit of $f_v(c)$ is 1 if and only if $\{v,b_i,c\}$ is a hyperedge of $H$. We claim that for every bijection $f\colon C\to C$ there is some $v\in H$ with $f_v = f$. First note that $f_a$ is the identity on $C$. Now, let $f\colon C\to C$ be a bijection. Extend it to a function $f'\colon B\cup C \to B\cup C$ by putting $f'(b) = b$ for every $b\in B$. Clearly, $f'$ is a partial automorphism of $G$ and hence extends to an automorphism $g$ of $H$. We have $f_{g(a)} = f$. As there are $m!$ such bijections $C\to C$ we need to have $\lvert H\rvert \geq m!$.
\end{proof}
Notice that if we fix a member of $B$, its hyperedges induce a bipartite graph on $A\cup C$. This was the key observation for our proof: Fixing all members of $B$ gave us $k$ different graphs which we were able to use to identify each vertex of $C$ uniquely.

This is essentially as far as one can get using this simple idea: There are only $2^{\mathcal O(n\log n)}$ partial automorphisms of any structure on $n$ vertices. Therefore, we pose the following question:
\begin{question}
Are there 3-uniform hypergraphs on $n$ vertices for arbitrarily large $n$ whose EPPA-witnesses need to have at least $c^{n^2}$ vertices for some constant $c>1$?
\end{question}
Proving a $2^{\omega(n\log(n))}$ lower bound would already be a significant progress (or even improving the base of the exponent in Observation~\ref{obs:3uniform}), even for $k$-uniform hypergraphs for $k$ not necessarily equal to 3.

\section*{Acknowledgments}
The authors would like to thank the anonymous referee, Martin Balko, David Hartman, Dugald Macpherson, András Pongrácz, and Robert Šámal for helpful comments.

\bibliography{ramsey.bib}
\end{document}